\documentclass[reqno,12pt]{amsart}

\parskip = \bigskipamount
\usepackage{amssymb}
\usepackage{amscd}
\usepackage{xcolor}
\usepackage[margin=1in]{geometry}

\newcommand{\R}{\mathbb{R}}

\newcommand{\inv}{^{-1}}
\newcommand{\bd}{\partial}
\newcommand{\eval}{\bigg\vert}

\newcommand{\RP}{\mathbb{RP}}
\newcommand{\vol}{\operatorname{Vol}}

\newcommand{\area}{\operatorname{Area}}
\newcommand{\eps}{\varepsilon}
\newcommand{\grad}{\nabla}
\newcommand{\lap}{\Delta}

\theoremstyle{plain}
\newtheorem{theorem}{Theorem}
\newtheorem{corollary}[theorem]{Corollary}
\newtheorem{prop}[theorem]{Proposition}
\newtheorem{lem}[theorem]{Lemma}

\newtheorem{thm}{Theorem}

\newtheorem{propA}[thm]{Proposition}

\newtheorem{rethm}{Theorem}

\newtheorem{repropA}[rethm]{Proposition}

\newtheorem{rerethm}{Theorem}

\theoremstyle{definition}
\newtheorem{defn}[theorem]{Definition}
\newtheorem{rem}[theorem]{Remark}

\newtheorem{question}{Question}
\newtheorem*{yamabe}{The Yamabe Problem} 

\title{On the Stability of the Yamabe Invariant of $S^3$}
\author{Liam Mazurowski}
\address{Department of Mathematics \\ Cornell University \\ Ithaca, NY, 14850}
\email{lmm334@cornell.edu}

\author{Xuan Yao}
\address{Department of Mathematics \\ Cornell University \\ Ithaca, NY, 14850}
\email{xy346@cornell.edu}

\begin{document}
\maketitle

\begin{abstract}
Let $g$ be a complete, asymptotically flat metric on $\R^3$ with vanishing scalar curvature. Moreover, assume that $(\R^3,g)$ supports a nearly Euclidean $L^2$ Sobolev inequality. We prove that $(\R^3,g)$ must be close to Euclidean space with respect to the $d_p$-distance defined by Lee-Naber-Neumayer. We then discuss some consequences for the stability of the Yamabe invariant of $S^3$. More precisely, we show that if such a manifold $(\R^3,g)$ carries a suitably normalized, positive solution to $\lap_g w + \lambda w^5 = 0$ then $w$ must be close, in a certain sense, to a conformal factor that transforms Euclidean space into a round sphere.
\end{abstract}

\section{Introduction} 

The classical uniformization theorem implies that every closed Riemann surface $\Sigma$ admits a conformal metric of constant Gaussian curvature.  In higher dimensions, there are known obstructions for a smooth manifold $M$ to admit a metric of constant sectional curvature or constant Ricci curvature. However, one can still ask if every smooth manifold $M$ admits a metric of constant scalar curvature. In the 1960s, Yamabe \cite{yamabe1960deformation} claimed that in fact every closed Riemannian manifold $(M^n,g)$ with dimension $n\ge 3$ admits a conformal metric $\bar g = e^{2\varphi}g$ with constant scalar curvature.  However, Trudinger \cite{trudinger1968remarks} later discovered a gap in Yamabe's proof.  The question of whether a given $(M,g)$ always admits a conformal metric of constant scalar curvature is now known as the Yamabe problem. 

\begin{yamabe} Does every closed Riemannian manifold $(M,g)$ admit a conformal metric with constant scalar curvature?
\end{yamabe}

To understand further developments in the Yamabe problem, it is important to introduce the so-called Yamabe quotient. Given a closed Riemannian manifold $(M^n,g)$, define 
\[
Y(M,[g]) = \inf\left\{\frac{\int_M a \vert \grad u\vert^2 + Ru^2\, dv}{\left(\int_M u^q \, dv\right)^{2/q}}: u\in C^\infty(M),\ u>0\right\}
\]
where 
\[
a = \frac{4(n-1)}{n-2} \quad\text{and}\quad q = \frac{2n}{n-2}
\]
are dimensional constants. 
It is known that $Y(M,[g])$ depends only on the conformal class of the metric $g$. 

Trudinger \cite{trudinger1968remarks} was able to fix the gap in Yamabe's proof under the extra assumption that $Y(M,[g])\le 0$.  In other words, Trudinger proved that if $Y(M,[g]) \le 0$, then there is a metric $\bar g$ conformal to $g$ with constant scalar curvature. Later, Aubin \cite{aubin1976equations} proved that the inequality
\begin{equation}
\label{ineq1}
Y(M^n,[g]) \le Y(S^n,[g_{\text{round}}])
\end{equation}
always holds. Moreover, Aubin %\cite{aubin1976equations} 
proved that if this inequality is strict, then there is a metric of constant scalar curvature in the conformal class of $g$. Aubin was able to show that strict inequality holds in (\ref{ineq1}) provided that $n \ge 6$ and $g$ is not locally conformally flat. Finally, Schoen \cite{schoen1984conformal} proved that strict inequality holds in all remaining cases and thus completed the affirmative resolution of the Yamabe problem. For more details on the Yamabe problem, see the excellent survey of Lee and Parker \cite{lee1987yamabe}.

It is worth examining Schoen's argument in more detail since it serves as important motivation for the results in this paper. From now on, we will focus attention on the 3-dimensional case. Let $(M^3,g)$ be a three dimensional Riemannian manifold. One can assume that $Y(M,[g]) > 0$ as otherwise there is nothing to prove. The conformal laplacian is the differential operator 
\[
L = -8\lap_M + R.
\]
Since $Y(M,[g]) > 0$, it is known that $L$ admits a positive Green's function $\Gamma$ with a pole at some fixed point $x\in M$.  Schoen observed that the manifold $(M-\{x\},\Gamma^4 g)$ is complete and asymptotically flat with vanishing scalar curvature.  We will call this manifold a {\it Yamabe blow up model} for $M$.   The positive mass theorem \cite{schoen1979proof} implies that $(M-\{x\},\Gamma^4 g)$ has positive ADM mass. This in turn gives crucial information about the asymptotics of $\Gamma$ near the pole. Schoen was able to exploit this information to construct a test function witnessing that $Y(M,[g])<Y(S^3,[g_{\text{round}}])$. 

Schoen's result shows that inequality (\ref{ineq1}) is {\it rigid}: if equality holds in (\ref{ineq1}) then $(M^3,g)$ is conformal to the round 3-sphere. Given this rigidity, it is natural to inquire about the {\it stability} of inequality (\ref{ineq1}). Namely, if $(M^3,g)$ almost achieves equality in (\ref{ineq1}), does it follow that $(M,g)$ is close to being conformal to a round 3-sphere in some sense? It is known that topological stability does not hold. Indeed, Kobayashi \cite{kobayashi1987scalar} and Schoen \cite{schoen2006variational} have constructed metrics $g_i$ on $S^2\times S^1$ such that $Y(S^2\times S^1,[g_i])\to Y(S^3,[g_{\text{round}}])$. Thus we will focus on the case where the underlying manifold $M$ is assumed to be $S^3$. 

\begin{question}[Stability of the Yamabe Invariant] 
\label{yamabestable} Assume that $g$ is a metric on $S^3$ such that 
\[
Y(S^3,[g_{\text{round}}]) - Y(S^3,[g]) < \eta
\]
for some small $\eta > 0$. Does it follow that some metric in the conformal class of $g$ is close to round in some sense?
\end{question}

In light of Schoen's argument, we expect that stability for the Yamabe invariant should be closely related to the stability of the blow up models. The blow up model for round $S^3$ is Euclidean $\R^3$. Now assume that $(S^3,g)$ nearly achieves equality in $(\ref{ineq1})$.  Then the blow up model $(S^3-\{x\},\Gamma^4 g)$ is a complete, asymptotically flat manifold with vanishing scalar curvature. Moreover, by conformal invariance, the near equality in (\ref{ineq1}) implies that the blow up model supports a nearly Euclidean $L^2$ Sobolev inequality. In other words, 
\[
\frac{\int_{S^3-\{x\}} \vert \grad u\vert^2\, dv}{\left(\int_{S^3-\{x\}} u^6\, dv\right)^{1/3}} \ge \Lambda - \delta 
\]
for all smooth functions $u$ vanishing in a neighborhood of $x$. Here 
$\Lambda = 3(\frac \pi 2)^{4/3}$ 
is the optimal $L^2$ Sobolev constant on Euclidean space, $\delta > 0$ is very small, and all geometric quantities in the integrals are computed with respect to $\Gamma^4 g$. Based on this, we pose the following stability question for Yamabe blow up models of $S^3$.  

\begin{question} [Stability of Yamabe Blow Up Models]
\label{blowupstable}
Assume that $(M,g)$ is a complete, asymptotically flat 3-manifold with vanishing scalar curvature. Assume that $M$ is topologically $\R^3$. Also assume that $(M,g)$ supports a nearly Euclidean $L^2$ Sobolev inequality:
\begin{equation}
\label{ineq2}
\frac{\int_M \vert \grad u\vert^2\, dv}{\left(\int_M u^6\, dv\right)^{1/3}} \ge \Lambda - \delta, \quad \forall u\in W^{1,2}(M).
\end{equation}
If $\delta > 0$ is small, then does $M$ have to be close to Euclidean space in some sense? 
\end{question}

{ We will begin by addressing Question \ref{blowupstable}, and will return later to discuss the extent to which an answer to Question \ref{blowupstable} gives an answer to Question \ref{yamabestable}.}  
To say more about Question \ref{blowupstable}, an important first step is to decide in what sense we can expect $M$ to be close to Euclidean. To motivate our choice, we first discuss what is known about a number of other stability questions related to scalar curvature. 

\subsection{Scalar Curvature Stability Results} Many {\it rigidity} results are known in the study of scalar curvature. We cannot hope to give an exhaustive list, but we mention the following three examples. 

\begin{itemize}

\item[(i)] {\bf The Positive Mass Theorem}:  A manifold $(M^n,g)$ is called asymptotically flat if, loosely speaking, there is a compact set $K\subset M$ such that $M- K$ is diffeomorphic to $\R^n-B$ and, in the coordinates given by this diffeomorphism, the metric on $M$ approaches Euclidean at a certain rate near infinity.  Motivated by physical considerations, Arnowitt-Deser-Misner \cite{arnowitt1961coordinate} associated to any such $M$ a number $m_{\text{ADM}}$ called the mass. The mass is a measure of how quickly the metric decays to Euclidean near infinity. The positive mass theorem asserts that for any asymptotically flat manifold $M$ with non-negative scalar curvature one has 
\[
m_{\text{ADM}}(M) \ge 0,
\]
and moreover, that equality holds if and only if $M$ is isometric to $\R^n$. This was first proven in dimension $3\le n\le 7$ by Schoen and Yau \cite{schoen1979proof} using minimal surface techniques. Later, Witten \cite{witten1981new} gave a proof that works in any dimension assuming that the manifold $M$ is spin. \\

\item[(ii)] {\bf Llarull's Theorem}: Llarull \cite{llarull1998sharp} proved that if $g$ is any metric on $S^n$ satisfying $g\ge g_{\text{round}}$ and $R\ge n(n-1)$ then $g$ must be round.  \\

\item[(iii)] {\bf The Geroch Conjecture}: Geroch conjectured that there is no metric of positive scalar curvature on a torus $T^n$. Moreover, any metric on $T^n$ with non-negative scalar curvature must be flat. In dimension 2, this is obvious from the Gauss-Bonnet theorem. In dimensions $3\le n \le 7$, Schoen and Yau \cite{schoen1987structure} again gave a proof using minimal surface methods. Gromov and Lawson \cite{gromov1983positive} gave a proof that works in any dimension using Dirac operator methods. 
\end{itemize}

Recently, a number of authors have investigated the associated {\it stability} questions. In the case of the positive mass theorem, one asks whether an asymptotically flat manifold $(M,g)$ with non-negative scalar curvature and ADM mass close to zero must be close to Euclidean space. 
In general, manifolds with non-negative scalar curvature may contain long thin splines. Even worse, they may contain ``other worlds,'' which are nearly arbitrary regions separated from the rest of manifold by thin necks.
These examples show that one cannot expect stability with respect to the Gromov-Hausdorff topology. 

This motivated Sormani and Wenger \cite{sormani2011intrinsic} to define the intrinsic flat topology, which effectively ignores splines. Lee and Sormani \cite{lee2014stability} proved stability of the positive mass theorem in the intrinsic flat topology for rotationally symmetric metrics with outermost minimal boundary. In our case, we will see that the Sobolev inequality (\ref{ineq2}) implies good isoperimetric control, and so splines and other worlds cannot form. Hence there is no need to use the intrinsic flat topology in our setting.
More recently, Dong and Song \cite{dong2023stability} proved stability for the positive mass theorem in the sense of Gromov-Hausdorff convergence away from a bad set whose boundary has small area. 
Again, this does not seem like a suitable mode of convergence to study the above question. Indeed, inequality (\ref{ineq2}) is scale invariant, and so we can always scale down $M$ so that all the interesting information is contained in a very small set. 

For Llarull's theorem, stability asks whether a metric $g$ on $S^n$ with $g\ge g_{\text{round}}$ and $R(g) \ge n(n-1)(1-\eps)$ must be close to round in some sense. Gromov deemed this the spherical stability problem. Allen-Bryden-Kazaras \cite{allen2023stability} have recently proven intrinsic flat stability for the 3-dimensional spherical stability problem, assuming some extra control over the diameter, volume, and isoperimetric constant. It is worth mentioning that harmonic functions (and more generally spacetime harmonic functions) play an important role in the stability results for the positive mass theorem and Llarull's theorem. In this regard, Dong and Song's result relies on the work of Bray-Kazaras-Khuri-Stern \cite{bray2022harmonic}, and the Allen-Bryden-Kararas result relies on the work of Hirsch-Kazaras-Khuri-Zhang \cite{hirsch2022rigid}. The proof of our main result will also use harmonic functions, albeit in a slightly different way. 

In the case of the Geroch conjecture, the stability problem asks whether a metric $g$  on $T^n$ with $R(g)\ge -\eps$ for some small $\eps >0$ must be close to a flat torus. The formation of splines and other worlds must still be taken into account in this setting. Sormani \cite{sormani2023conjectures} conjectured that if $g$ additionally satisfies the so-called min-A condition, which is a weak form of non-collapsing, then $(T^n,g)$ should be close to flat in the intrinsic flat topology. However, Lee-Naber-Neumayer \cite{lee2023dp} showed that in dimension $n\ge 4$ there are tori $(T^n,g)$ with unit volume  and $R(g) \ge -\eps$, but with very tiny diameter, even in the presence of a strong non-collapsing condition.  Similar examples were then constructed in dimension 3 by Kazaras and Xu \cite{kazaras2023drawstrings}, who refer to them as ``drawstrings.'' 

In light of this phenomenon, Lee-Naber-Neumayer \cite{lee2023dp} introduced the notion of $d_p$ distance.  For a Riemannian manifold $(M^n,g)$ and a fixed value $p>n$, this distance is defined as follows. 
\begin{defn}

The $d_p$ distance between two points $x,y\in M$ is given by 
\[
d_p(x,y) = \sup\left\{\vert f(x)-f(y)\vert: \int_M \vert \grad f\vert^p\, dv \le 1,\ f\in W^{1,p}_{\text{loc}}(M) \right\}. 
\]
This is in fact an honest distance function and, for fixed $M$, it converges to the usual geodesic distance on $M$ as $p\to \infty$. Let $\mathcal B_{p,g}(x,r)$ denote the $d_p$ ball of radius $r$ in $M$ centered at $x$.  
\end{defn}

 While the above drawstring-type examples do not behave well with respect to Gromov-Hausdorff convergence, Lee-Naber-Neumayer showed that they do have nice $d_p$-limits.  More generally, they proved that the $d_p$ distance provides a good framework for studying the convergence of manifolds with almost non-negative scalar curvature and almost Euclidean entropy. Their results imply, for example, that a sequence of complete manifolds $(M_i,g_i)$ with $R_i \ge -\eps_i$ and  $\nu(g_i,2)\ge -\eps_i$ where $\eps_i\to 0$ converges in the $d_p$ sense to Euclidean space for all large $p$. Here $\nu$ denotes Perelman's $\nu$-functional and the condition $\nu(g_i,2)\ge -\eps_i$ represents a kind of strong non-collapsing condition. Since Perelman's $\nu$-functional is closely related to  optimal Sobolev inequalities, it therefore seems reasonable to expect $d_p$ convergence in our setting. 

\subsection{Main Results} We first prove a stability type theorem for Yamabe blow up models of $S^3$ with respect to $d_p$ convergence. 

\begin{defn}
Let $(M_i,g_i)$ be a sequence of complete, asymptotically flat 3-manifolds such that each $M_i$ is topologically $\R^3$. Fix a number $p\in (3,\infty)$. Then $(M_i,g_i)$ converges to Euclidean space in the $d_p$ sense if for all $x_i \in M_i$ and all fixed $r > 0$ we have 
\[
d_{\text{GH}}\bigg((\mathcal B_{p,g_i}(x_i,r),d_{p,g_i}),(\mathcal B_{p,g_{\text{euc}}}(0,r), d_{p,\text{euc}})\bigg) \to 0, \quad \text{ as } i\to\infty
\]
and, moreover, 
\[
\vol_{g_i}(\mathcal B_{p,g_i}(x_i,r)) \to { \vol_{g_{\text{euc}}}(\mathcal B_{p,g_{\text{euc}}}(0,r))}, \quad\text{ as } i\to \infty. 
\]
Here $d_{\text{GH}}$ denotes the Gromov-Hausdorff distance between metric spaces.  Note that we do not need to introduce $\text{Cov}(x,N)$ as in \cite{lee2023dp} Definition 2.44 since the spaces we work with are asymptotically flat, and therefore $d_p$ balls have compact closure. 
\end{defn} 

\begin{thm}
\label{main}
Assume that $(M_i,g_i)$ is a sequence of complete, asymptotically flat 3-manifolds with vanishing scalar curvature. Assume that each $M_i$ is topologically $\R^3$. Further suppose that 
\[
\inf\left\{\frac{\int_{M_i} \vert \grad u\vert^2\, dv}{\left(\int_{M_i} u^6\, dv\right)^{1/3}}: u\in W^{1,2}(M_i)\right\} \ge \Lambda - \delta_i
\]
where $\delta_i\to 0$. Then $M_i$ converges to Euclidean space in the $d_p$ sense for all $p\in(3,\infty)$. 
\end{thm}

\begin{rem}
We do not know whether the $d_p$ convergence in Theorem \ref{main} is optimal. The counter-examples to Gromov-Hausdorff convergence constructed in \cite{lee2023dp} all have small but negative scalar curvature at some points, and hence they do not occur as Yamabe blow up models. 
\end{rem}

\begin{rem}
In principle, one also expects the stability of Yamabe blow-up models for $S^n$ when $n\ge 4$. Theorem \ref{main} is restricted to dimension three because our argument relies in a crucial way on certain monotonicity formulas that are proven using the Gauss-Bonnet theorem. 
\end{rem}

Of course, one would ultimately like to prove stability for the Yamabe invariant itself, and not just the blow up models.  To this end, note that a Yamabe blow-up model $(M^3,g)$ carries a conformal factor that ``undoes'' the blow up procedure. We are able to show that these conformal factors on the blow up models converge to a conformal factor that transforms Euclidean $\R^3$ into round $S^3$. 

More precisely, assume that $\bar g_i$ is a sequence of metrics on $S^3$ and that 
\[
Y(S^3,[g_{\text{round}}]) - Y(S^3,[\bar g_i]) \to 0
\]
as $i\to \infty$. By the resolution of the Yamabe problem, after replacing $\bar g_i$ by a conformal metric if necessary, we can suppose that $\vol(\bar g_i) = 2\pi^2$ and $R(\bar g_i) \equiv s_i$ is constant. Moreover, in this case, one has $s_i \to 6$. Fix a point $x\in S^3$ and consider the blow up models 
\[
(M_i,g_i) = (S^3-\{x\}, \Gamma_i^4 \bar g_i),
\]
where the Green's functions $\Gamma_i$ for the conformal Laplacian are normalized so that $\min_{M_i} \Gamma_i = 1$.  
Then $w_i = \Gamma_i\inv$ solves 
\[
\lap_{g_i} w_i + \lambda_i w_i^5 = 0
\]
on $M_i$, where $\lambda_i = \frac{s_i}{8}$ is a positive constant. 
Moreover, we have 
\[
\|w_i\|_{L^\infty(M_i)} = 1 \text{ and } \|w_i\|_{L^6(M_i)}^6 = 2\pi^2
\]
and $\lambda_i \to \frac 3 4$. 

Theorem \ref{main} implies that the blow up models $(M_i,g_i)$ converge to Euclidean space in the $d_p$ sense. Let $x_i$ be a point where $w_i(x_i) = 1$. The $d_p$ theory gives the existence of ``nice'' diffeomorphisms $\psi_i\colon \Omega_i \to B(0,r_i)$ with $\psi_i(x_i) = 0$. Here $\Omega_i$ is a neighborhood of $x_i$ in $M_i$, and $B(0,r_i)$ is a ball of radius $r_i$ in $\R^3$, and $r_i\to \infty$.  The map $\psi_i$ is, in particular, an $\eps_i$-Gromov-Hausdorff approximation in the $d_p$-distance. See Section \ref{further} for more details. Define $\tilde w_i = w_i\circ \psi_i\inv$ so that $\tilde w_i\colon B(0,r_i)\to \R$.  Then we are able to prove that $\tilde w_i$ converges to one of the canonical conformal factors that transforms $\R^3$ into a round sphere.  

\begin{thm}
\label{main2}
Assume that $(M_i,g_i)$ is a sequence of complete, asymptotically flat 3-manifolds with vanishing scalar curvature. Assume that each $M_i$ is topologically $\R^3$, and that the optimal constant in the $L^2$ Sobolev inequality on $M_i$ is approaching the Euclidean constant. Suppose that $M_i$ carries a positive solution to 
\[
\lap_{g_i} w_i + \lambda_i w_i^5 = 0
\]
with 
\[
\|w_i\|_{L^\infty(M_i)} = 1 \text{ and } \|w_i\|_{L^6(M_i)}^6 = 2\pi^2 \text{ and } \lambda_i \to \frac 3 4. 
\]
Let $\psi_i$ be the diffeomorphisms described above and let $\tilde w_i = w_i \circ \psi_i\inv$. Then for any $q < 2$ and any $1 \le s < \infty$, the functions $\tilde w_i$ converge weakly in $W^{1,q}_{\text{loc}}(\R^3)$ and strongly in $L^s_{\text{loc}}(\R^3)$ to the function
\[
\tilde w(x) =  \sqrt{\frac{4}{4+\vert x\vert^2}}.
\]
\end{thm}

\subsection{Sketch of Proof}

In the remainder of the introduction, we outline the proof of the main { theorems. We begin with Theorem \ref{main}.} Assume that $(M,g)$ is a complete, asymptotically flat 3-manifold with vanishing scalar curvature. Assume that $M$ is topologically $\R^3$ and that 
\[
\inf\left\{\frac{\int_{M} \vert \grad u\vert^2\, dv}{\left(\int_{M} u^6\, dv\right)^{1/3}}: u\in W^{1,2}(M)\right\} \ge \Lambda - \delta
\]
for some small $\delta > 0$. 

Recall that a crucial step in Schoen's resolution of the Yamabe problem was to invoke the positive mass theorem on $M$. 
Naively, one might hope to show that if $\delta$ is small then the ADM mass of $M$ is small, and therefore that one can apply stability for the positive mass theorem. However, this approach seems unlikely to succeed since the Sobolev inequality is scale invariant while the ADM mass is not. Similarly, given a surface $\Sigma$ embedded in $M$, one cannot expect to control the Hawking mass 
\[
m_H(\Sigma) = {\frac{\area(\Sigma)^{1/2}}{(16\pi)^{3/2}}}\left(16\pi - \int_\Sigma H^2\, da\right)
\]
as this quantity is also not scale invariant. Thus we focus our attention instead on the Willmore energy 
\[
\mathcal W(\Sigma) = \int_\Sigma H^2\, da,
\] 
which is scale invariant. 

Willmore \cite{willmore1965note} proved that in $\R^3$ all connected, embedded surfaces $\Sigma$ satisfy $\mathcal W(\Sigma) \ge 16\pi$. Recently, Agostiniani-Mazzieri \cite{agostiniani2020monotonicity} and Agostiniani-Fogagnolo-Mazzieri \cite{agostiniani2020sharp} proved similar Willmore-type inequalities in complete manifolds with non-negative Ricci curvature. As a first step, we show that a Willmore-type inequality holds in $M$ with some small error.

\begin{propA} 
\label{willmore} Assume that $(M,g)$ is a complete, asymptotically flat manifold with vanishing scalar curvature.  Assume that $M$ is topologically $\R^3$ and that
\[
\inf \left\{\frac{\int \vert \grad u\vert^2\, dv}{\left(\int u^6\, dv\right)^{1/3}}: u\in W^{1,2}(M)\right\} \ge \Lambda - \delta
\]
where $\delta > 0$ is small. Then there is an $\eps = \eps(\delta)$ such that 
\[
\mathcal W(\Sigma) = \int_\Sigma H^2\, da \ge 16\pi(1 - \eps)^2
\]
for all compact, connected, embedded surfaces $\Sigma$ in $(M,g)$. Here $\eps(\delta)\to 0$ as $\delta \to 0$. 
\end{propA} 

The proof of Proposition \ref{willmore} is inspired by an argument of Bray and Neves \cite{bray2004classification}. It is also morally related to the P\'olya-Szeg\H{o} principle and the Faber-Krahn inequality. Bray and Neves proved that 
\begin{equation}
Y(\RP^3,[g]) \le Y(\RP^3,[g_{\text{round}}])
\end{equation}
for all metrics $g$ on $\RP^3$. 
Their proof uses the level sets of weak inverse mean curvature flow to transfer an optimal test function from the standard blow up model for $\RP^3$ to an arbitrary blow up model. In \cite{mazurowski2023yamabe}, the authors showed that the optimal test function can also be transferred using the level sets of a harmonic function. Here we observe that the argument from \cite{mazurowski2023yamabe} can be modified to estimate the Willmore energy of all connected surfaces $\Sigma$ embedded in $M$. 
We note that this step relies on certain monotonicity formulas for harmonic functions derived by Miao \cite{miao2023mass}. These monotonicity formulas ultimately rely on the Gauss-Bonnet Theorem, and hence our proof does not easily generalize to higher dimensions. This application of Miao's monotonicity formulas also uses the fact that $M$ is topologically $\R^3$ in a crucial way. 

After proving Proposition \ref{willmore}, we then leverage the control over the Willmore functional to understand the isoperimetric profile of $M$. Let 
\[
I_M(v) = \inf\{\area(\bd \Omega): \Omega\subset M,\ \vol(\Omega) = v\}
\]
be the isoperimetric profile of $M$. Applying the Willmore-type inequality to isoperimetric surfaces in $M$, it is possible to derive a differential inequality satisfied by $I_M$. This ultimately yields a comparison between the isoperimetric profile of $M$ and the Euclidean isoperimetric profile $I_{\text{euc}}$.

\begin{propA}
\label{iso}
Assume that $(M,g)$ is a complete, asymptotically flat manifold with vanishing scalar curvature.  Assume that $M$ is topologically $\R^3$ and that 
\[
\inf \left\{\frac{\int \vert \grad u\vert^2\, dv}{\left(\int u^6\, dv\right)^{1/3}}: u\in W^{1,2}(M)\right\} \ge \Lambda - \delta
\]
where $\delta > 0$ is small. Then there is an $\eps = \eps(\delta)$ such that 
\[
I_M(v) \ge \left(1-2\eps\right)^{2/3} I_{\operatorname{euc}}(v)
\]
for all $v > 0$. Here $\eps(\delta)\to 0$ as $\delta\to 0$. 
\end{propA} 

Finally, to prove Theorem \ref{main}, we show that this isoperimetric control allows us to apply the $d_p$ theory of Lee-Naber-Neumayer \cite{lee2023dp}. 

{ Next, to prove Theorem \ref{main2}, we use the fact that ``$d_p$ convergence preserves the $W^{1,2}$ analysis.'' Hence, given functions $w_i$ and diffeomorphisms $\psi_i$ as in the statement of Theorem \ref{main2}, we can show that the functions $\tilde w_i = w_i\circ \psi_i\inv$ converge to a weak solution $\tilde w$ of the equation 
\begin{equation}
\label{w-pde}
\lap \tilde w + \frac{3}{4} \tilde w^5 = 0
\end{equation}
on $\R^3$. We then use a Moser iteration argument to show that the normalization $\|w_i\|_{L^\infty(M_i)} = 1$ prevents concentration, and therefore that $\tilde w$ is non-zero. Finally, we use the classification of $L^2$ Sobolev minimizers on $\R^3$ to conclude that $\tilde w$ has the desired form.} 

\subsection{Organization} The remainder of the paper is organized as follows. In Section \ref{sec:willmore}, we prove Proposition \ref{willmore} on the nearly optimal Willmore inequality. In Section \ref{sec:iso}, we prove Proposition \ref{iso} on the isoperimetric profile of the blow up models. In Section \ref{sec:dp}, we obtain $d_p$ convergence and complete the proof of Theorem \ref{main}. 
Finally, in Section \ref{further}, we show the convergence of the conformal factors and complete the proof of Theorem \ref{main2}. 

\subsection{Acknowledgements}
The authors would like to thank Xin Zhou for his tremendous support and encouragement. They would also like to thank Man-Chun Lee and Gaoming Wang for helpful discussions related to this work. X.Y. would like to thank Sun-Yung Alice Chang and Paul Yang for their helpful discussions related to this work. L.M. acknowledges the support of an AMS Simons Travel Grant. X.Y. is supported by NSF grant DMS-1945178.

\section{A Willmore type inequality} 

\label{sec:willmore}

This section is dedicated to the proof of Proposition \ref{willmore}. The proof will use some of Miao's monotonicity formulas for harmonic functions \cite{miao2023mass}. In these monotonicity formulas, it is essential that the level sets of the harmonic function under consideration are connected. In our situation, this follows from the fact that $M$ is topologically $\R^3$. 

\begin{lem}
\label{lem:connect}
Assume that $(M,g)$ is a complete, asymptotically flat manifold that is topologically $\R^3$. Let $\Sigma$ be a compact, connected, embedded surface in $M$ and let $\Omega$ be the compact region enclosed by $\Sigma$. Let $\varphi$ be the solution to 
\[
\begin{cases}
\lap \varphi =0, &\text{in } M - \Omega\\
\varphi = 1, &\text{on } \Sigma,\\
\varphi\to 0, &\text{at infinity}. 
\end{cases}
\]
Then all regular level sets of $\varphi$ are connected. 
\end{lem}

\begin{proof}
Consider some regular value $0 < t < 1$ of $\varphi$. Suppose for contradiction that $\{\varphi = t\}$ is not connected. Then we can find two distinct compact, connected, embedded surfaces $\Gamma_1$ and $\Gamma_2$ contained in $\{\varphi=t\}$. Since $M$ is topologically $\R^3$, it follows that $\Gamma_i$ bounds a compact region $\Omega_i$ for $i=1,2$. 
Note that both regions $\Omega_i$ must intersect $\Omega$. Indeed, if $\Omega_i$ did not intersect $\Omega$ then $\varphi$ would have either an interior minimum or an interior maximum on $\Omega_i$ and therefore $\varphi$ would be constant.  Hence $\Omega_i$ intersects $\Omega$. In fact, this implies that $\Omega_i$ entirely contains $\Omega$ since $\Gamma_i$ is disjoint from $\Sigma$. 

Next, since $\Gamma_1$ and $\Gamma_2$ are disjoint, either $\Omega_1$ and $\Omega_2$ are disjoint, or $\Omega_1$ and $\Omega_2$ are nested. The former possibility is impossible since $\Omega_1$ and $\Omega_2$ both enclose $\Sigma$. Therefore $\Omega_1$ and $\Omega_2$ are nested. Without loss of generality, we can suppose that $\Omega_1 \subset \Omega_2$. But then $\Omega_2 \setminus \Omega_1$ does not intersect $\Omega$. Hence $\varphi$ has either an interior minimum or an interior maximum in $\Omega_2 \setminus \Omega_1$. This implies that $\varphi$ is constant, which is a contradiction.  This lemma follows.
\end{proof} 

Now we can give the proof of Proposition \ref{willmore}.  The proof is a more quantitative version of the argument in \cite{mazurowski2023yamabe}.

\stepcounter{rethm}
\stepcounter{rethm}
\begin{repropA} Assume that $(M,g)$ is a complete, asymptotically flat manifold with vanishing scalar curvature.  Assume that $M$ is topologically $\R^3$ and that 
\[
\inf \left\{\frac{\int \vert \grad u\vert^2\, dv}{\left(\int u^6\, dv\right)^{1/3}}: u\in W^{1,2}(M)\right\} \ge \Lambda - \delta
\]
where $\delta > 0$ is small. Then there is an $\eps = \eps(\delta)$ such that 
\[
\mathcal W(\Sigma) = \int_\Sigma H^2\, da \ge 16\pi(1 - \eps)^2
\]
for all compact, connected, embedded surfaces $\Sigma$ in $(M,g)$. Here $\eps(\delta)\to 0$ as $\delta \to 0$. 
\end{repropA}

\begin{proof}
Consider a compact, connected, embedded surface $\Sigma$ in $M$. Note that $\Sigma$ bounds a compact region $\Omega$.  If $\mathcal W(\Sigma) \ge 16\pi$ then there is nothing to prove. So assume that 
\[
\int_\Sigma H^2\, da =  16\pi (1 - \eta)^2
\]
for some $0< \eta \le 1$. 
Since both the Sobolev inequality and the Willmore energy are scale invariant, we can assume without loss of generality that $\vol(\Omega) = \frac{4\pi}{3}$.  The functions 
\[
s_\beta(r) = \sqrt{\frac{\beta^2 + 1}{\beta^2 + r^2}}
\]
all achieve equality in the Euclidean Sobolev inequality (see \cite{lee1987yamabe} Section 3). Moreover, $s_\beta$ converges smoothly to 1 on compact subsets of $\R^3$ as $\beta \to \infty$. 
Therefore, 
it is possible to select $\beta$ large enough that 
\[
\Lambda \ge \frac{\int_{\R^3 - B_1} \vert \grad s_\beta\vert^2\, dv}{\left(\int_{\R^3-B_1} s_\beta^6\, dv\right)^{1/3}} - \delta. 
\]
In fact, we will need to quantitatively relate $\beta$ and $\delta$. To this end, one may explicitly evaluate the above integrals in polar coordinates to see that 
\begin{align*}
\int_{\R^3-B_1} \vert \grad s_\beta\vert^2\, dv &= \int_1^\infty 4\pi r^2\frac{ (1+\beta^2)r^2}{(\beta^2+r^2)^3}\, dr\\
&= \frac{\pi}{2}(1+\beta^2)\left(\frac{-3\beta^2 r - 5r^3}{(\beta^2+r^2)^2}+\frac{3\arctan(\frac r \beta)}{\beta}\right)\eval_1^\infty\\
& = \frac{\pi}{2}\left(\frac{5\beta + 3\beta^3 + 3(1+\beta^2)^2 \arctan(\beta)}{\beta(1+\beta^2)}\right) = \frac{3\pi^2}{4}\beta + \frac{3\pi^2}{4}\beta^{-1} + O(\beta^{-2}),
\end{align*}
and that 
\begin{align*}
\int_{\R^3-B_1} s_\beta^6\, dv &= \int_1^\infty 4\pi r^2\left(\frac{\beta^2 + 1}{\beta^2+r^2}\right)^3\, dr \\
& = \frac{\pi(1+\beta^2)^3}{2\beta^3(\beta^2+r^2)^2}\left(-\beta^3 r + \beta r^3 + (\beta^2+r^2)^2\arctan\left(\frac r \beta\right) \right)\eval_1^\infty\\
&= \frac{\pi}{2}\left(\frac{-\beta + \beta^5 + (1+\beta^2)^3\arctan(\beta)}{\beta^3}\right) = \frac{\pi^2}{4}\beta^3 + \frac{3\pi^2}{4}\beta - \frac{4\pi}{3} + O(\beta^{-1}). 
\end{align*}
Therefore one has  
\begin{align*}
\frac{\int_{\R^3 - B_1} \vert \grad s_\beta\vert^2\, dv}{\left(\int_{\R^3-B_1} s_\beta^6\, dv\right)^{1/3}} 
= \Lambda + \frac{2^{8/3} \pi^{1/3}}{3}\beta^{-3} + O(\beta^{-4}),
\end{align*}
and so it suffices to choose $\beta = O(\delta^{-1/3})$. 

Let $\varphi$ be the capacitary potential for $\Sigma$, i.e., the solution to 
\[
\begin{cases}
\lap \varphi = 0, &\text{in } M - \Omega\\
\varphi=1, &\text{on } \Sigma\\
\varphi\to 0, &\text{at infinity}. 
\end{cases}
\]
Note that $\bar \varphi(r) = 1/r$ is the capacitary potential for the Euclidean ball $B_1$.  Let $w = -\log \varphi$ and $\bar w = -\log \bar \varphi = \log r$. 
Now define $f_\beta\colon [0,\infty)\to \R$ so that $s_\beta(r) = f_\beta\circ \bar w(r)$ for $r\ge 1$. Explicitly, one has
\[
f_\beta(t) = \sqrt{\frac{\beta^2+1}{\beta^2+e^{2t}}}.
\]
Then define $s:M\to \R$ by 
\[
s = 
\begin{cases}
1, &\text{on } \Omega,\\
f_\beta\circ w, &\text{on } M- \Omega. 
\end{cases}
\]
We now show that plugging $s$ into the almost Euclidean Sobolev inequality yields an estimate for $\mathcal W(\Sigma)$. 

Define 
\[
W(t) = \int_{\{w=t\}} \vert \grad w\vert^2\, da. 
\]
Note that Lemma \ref{lem:connect} implies that all level sets of $\varphi$ and hence $w$ are connected.  Therefore, Proposition 3.1 in \cite{mazurowski2023yamabe} implies that 
\[
W(t) \le \left[e^{-t}\sqrt{W(0)} + (1-e^{-t})\sqrt{4\pi}\right]^2
\]
for all $t\ge 0$. Moreover, by a result of Miao \cite[Corollary 7.1]{miao2023mass}, we get that 
\[
{\sqrt{W(0)}} \le \sqrt{\pi} + \frac{1}{4}\left(\int_\Sigma H^2\, da\right)^{1/2} = \sqrt{\pi}  + \sqrt{\pi}(1-\eta) = (2-\eta)\sqrt{\pi}. 
\]
Combining the previous two inequalities, it follows that
\[
W(t) \le \pi \left[(2-\eta) e^{-t} + 2(1-e^{-t})\right]^2 = \pi(2 - \eta e^{-t})^2
\]
for all $t \ge 0$. 

Now observe that 
\[
\int_{M-\Omega} \vert \grad s\vert^2 \, dv = \int_0^\infty f_\beta'(t)^2 \left( \int_{\{w=t\}} \vert \grad w\vert\, da\, \right) dt = C \int_0^\infty f_\beta'(t)^2 e^t \,dt 
\]
where $C = \int_{\Sigma} \vert \grad w\vert\, da$. Likewise one has  
\[
\int_{\R^3-B_1} \vert  \grad  s_\beta\vert^2 \, dv  = \int_0^\infty f_\beta'(t)^2\left(\int_{\{\bar w=t\}} \vert  \grad \bar w\vert\, da\right)\, dt = 4\pi \int_0^\infty f_\beta'(t)^2 e^t\, dt.
\]
We also have 
\begin{align*}
\int_{M-\Omega} s^6\, dv &= \int_0^\infty f_\beta(t)^6 \left(\int_{\{w=t\}} \vert \grad w\vert^{-1}\, da\right)\, dt \\
&\ge \int_0^\infty f_\beta(t)^6 \left(\int_{\{w=t\}} \vert \grad w\vert^2\, da\right)^{-2}\left(\int_{\{w=t\}} \vert \grad w\vert\, da\right)^3\, dt\\
&\ge \pi^{-2} C^3 \int_0^\infty f_\beta(t)^6 e^{3t}  (2-\eta e^{-t})^{-4}\, dt
\end{align*}
and likewise
\[
\int_{\R^3-B_1}  s_\beta^6 \, dv = 2^{-4} \pi^{-2} (4\pi)^3 \int_0^\infty f_\beta(t)^6 e^{3t} \, dt. 
\]
Thus one obtains
\[
\Lambda + \delta \ge  \frac{\int_0^\infty f_\beta'(t)^2 e^t\, dt}{\left(2^{-4}\pi^{-2} \int_0^\infty f_\beta(t)^6 e^{3t}\, dt\right)^{1/3}} 
\]
and also
\[
\Lambda - \delta \le \frac{\int_M \vert \grad s\vert^2\, dv}{\left(\int_{M} s^6\, dv\right)^{1/3}} \le  \frac{\int_0^\infty f_\beta'(t)^2 e^t\, dt}{\left(\pi^{-2} \int_0^\infty f_\beta(t)^6 e^{3t} (2-\eta e^{-t})^{-4}\, dt\right)^{1/3}}.
\]
It follows that
\[
0 \le \frac{\int_0^\infty f_\beta'(t)^2 e^t\, dt}{\left(2^{-4}\pi^{-2} \int_0^\infty f_\beta(t)^6 e^{3t}\, dt\right)^{1/3}} -  \frac{\int_0^\infty f_\beta'(t)^2 e^t\, dt}{\left(\pi^{-2} \int_0^\infty f_\beta(t)^6 e^{3t} (2-\eta e^{-t})^{-4}\, dt\right)^{1/3}} \le 2\delta. 
\]
For simplicity, rewrite this as 
\[
0 \le \frac{a}{b}-\frac{a}{c} \le 2\delta. 
\]
Then we have 
\[
0 \le  1 - \frac{b}{c} \le  \frac{2\delta b}{a}\le C\delta. 
\]
It remains to translate this into a bound on $\eta$. 

As above, one can more or less explicitly evaluate the integrals. One has 
\begin{align*}
b^3 &= 2^{-4}\pi^{-2} \int_0^\infty f_\beta(t)^6 e^{3t}\, dt = \frac{1}{(4\pi)^3} \int_{\R^3-B_1}  s_\beta^6 \, dv \\
&= \frac{\beta^5 - \beta + (1+\beta^2)^3\arctan(\beta)}{128\pi^2 \beta^3} =
\frac{1}{256\pi}\beta^3  + O(\beta), 
\end{align*}
and therefore 
$
C_1 \beta^3 \le b^3 \le C_2\beta^3
$
for some absolute constants $C_1, C_2$.
Note also that
\[
(2-\eta e^{-t})^{-4} \ge 2^{-4}\left(1+\frac{\eta}{2} e^{-t}\right)
\]
for all $t\ge0$. 
One computes that 
\begin{align*}
2^{-4}\pi^{-2}\int_0^\infty f_\beta(t)^6 e^{3t} \left(1+\frac{\eta}{2} e^{-t}\right)\, dt &= b^3 - \frac{(1+\beta^2)^3 \eta}{128\pi^2 (\beta^2+e^{2t})^2}\eval_{0}^\infty = b^3 + \frac{(1+\beta^2)\eta}{128\pi^2}
\end{align*}
and hence that 
\[
c^3 \ge b^3 + \frac{(1+\beta^2)\eta}{128\pi^2}.
\]
Thus, for a constant $C$ that is allowed to vary from place to place but is independent of $\delta$, $\beta$, and $\eta$, we get that 
\[
\frac{b^3}{c^3} \le \frac{b^3}{b^3 + C (1+\beta^2)\eta} = \frac{1}{1+ C(\eta/\beta)} \le 1 - \frac{C\eta}{\beta}. 
\]
It follows that
\[
\frac{b}{c} \le \left(1-\frac{C\eta}{\beta}\right)^{1/3} \le 1 - \frac{C\eta}{\beta} 
\]
and therefore that
\[
C\delta \ge 1-\frac{b}{c} \ge \frac{C\eta}{\beta}.
\]
Since $\delta = O(\beta^{-3})$, this implies that $\eta = O(\delta^{2/3})$ which goes to 0 as $\delta \to 0$. 
\end{proof}

\section{Isoperimetric Control} 

\label{sec:iso}

Next we exploit the Willmore-type inequality to gain isoperimetric control. As above, assume that $(M,g)$ is a complete, asymptotically flat manifold with vanishing scalar curvature. Moreover, assume that $M$ is topologically $\R^3$ and that $M$ has a nearly Euclidean $L^2$ Sobolev inequality.
Let 
\[
I_M(v) = \inf\{\area(\bd \Omega): \Omega \subset M,\ \vol(\Omega) = v\}
\]
be the isoperimetric profile for $M$. In the next two propositions, we collect some well-known properties of the isoperimetric profile of a 3-dimensional asymptotically flat manifold with non-negative scalar curvature.

The first proposition gives the existence of an isoperimetric region for all volumes $v > 0$. This was proven by Carlotto-Chodosh-Eichmair in \cite{carlotto2016effective}, making use of Shi's isoperimetric inequality \cite{shi2016isoperimetric}. 

\begin{prop}
For each $v > 0$, there exists an isoperimetric region $\Omega\subset M$ such that $\vol(\Omega) = v$ and $\area(\bd \Omega) = I(v)$.  Moreover, $\bd \Omega$ is a (possibly disconnected) constant mean curvature surface. 
\end{prop}

The continuity and differentiability properties of $I_M$ can be proved as in \cite{flores2014continuity}. 

\begin{prop}
The isoperimetric profile function $I_M$ is absolutely continuous and strictly increasing. The left and right derivatives $(I_M)_-'$ and $(I_M)_+'$ exist at every point and are equal almost everywhere. If $\Sigma = \bd \Omega$ is any isoperimetric region enclosing volume $v$ then one has 
\[
(I_M)'_-(v) \ge H(\Sigma) \ge (I_M)'_+(v).
\]
In particular, if $I_M$ is differentiable at $v$, then $I_M'(v) = H(\Sigma)$. 
\end{prop}

Next, we compute a differential inequality satisfied by $I_M$. 

\begin{prop}
\label{diffineq}
Assume that $\mathcal W(\Sigma) \ge 16\pi(1-\eps)^2$ for all connected, embedded surfaces in $M$. Then the differential inequality 
\[
(I_M)'_-(v) > \left(1-2{\eps}\right)\sqrt{\frac{16\pi}{I_M(v)}}
\]
holds for all $v > 0$.
\end{prop} 

\begin{proof} 
Fix some $v>0$ and let $\Sigma_v$ be an isoperimetric surface enclosing volume $v$. Then $I_M(v) = \area(\Sigma_v)$ and $(I_M)'_-(v) \ge H(\Sigma_v)$. Let $\Gamma_v$ be some connected component of $\Sigma_v$. By the Willmore bound, we get that 
\[
I_M(v) (I_M)'_-(v)^2 \ge \area(\Sigma_v) H(\Sigma_v)^2 \ge \area(\Gamma_v)H(\Gamma_v)^2 \ge 16\pi(1-\eps)^2.
\]
Thus we have 
\[
(I_M)'_-(v) \ge (1-\eps)\sqrt{\frac{16\pi}{I_M(v)}} > (1-2\eps)\sqrt{\frac{16\pi}{I_M(v)}},
\]
as needed. 
\end{proof}

Now observe that the solution $y$ to 
\[
y'(v) = (1-2\eps)\sqrt{\frac{16\pi}{y(v)}}, \quad y(0) = 0
\]
is given by
\[
y(v) = (36\pi)^{1/3} (1-2\eps)^{2/3} v^{2/3} = (1-2\eps)^{2/3} I_{\text{euc}}(v)
\]
where $I_{\text{euc}}$ is the Euclidean isoperimetric profile. 
Hence the previous differential inequality for $I_M$ gives rise to a comparison with the Euclidean isoperimetric profile. We can now give the proof of Proposition \ref{iso}. 

\begin{repropA}
Assume that $(M,g)$ is a complete, asymptotically flat manifold with vanishing scalar curvature.  Assume that $M$ is topologically $\R^3$ and that
\[
\inf \left\{\frac{\int \vert \grad u\vert^2\, dv}{\left(\int u^6\, dv\right)^{1/3}}: u\in W^{1,2}(M)\right\} \ge \Lambda - \delta
\]
where $\delta > 0$ is small. Then there is an $\eps = \eps(\delta)$ such that 
\[
I_M(v) \ge \left(1-2\eps\right)^{2/3} I_{\text{euc}}(v)
\]
for all $v > 0$. Here $\eps(\delta)\to 0$ as $\delta\to 0$. 
\end{repropA}

\begin{proof}
Choose $\eps = \eps(\delta)$ according to Proposition \ref{willmore}, so that 
\begin{equation}
\label{ineq3}
(I_M)_-'(v) > (1-2\eps) \sqrt{\frac{16\pi}{I_M(v)}}
\end{equation}
holds for all $v > 0$ by Proposition \ref{diffineq}. By the asymptotics of the isoperimetric profile for small volumes, we know that $I_M(v) > (1-2\eps)^{2/3}I_{\text{euc}}(v)$ for all sufficiently small $v > 0$. Suppose for contradiction that for some $v_1 > 0$ one has $I_M(v_1) < (1-2\eps)^{2/3}I_{euc}(v_1)$. Then let 
\[
v_0 = \inf\{v> 0: I_M(v)\le (1-2\eps)^{2/3}I_{\text{euc}}(v)\} > 0. 
\]
Note that $I_M(v_0) = (1-2\eps)^{2/3} I_{\text{euc}}(v_0)$ and 
\[
(I_M)_-'(v_0) \le (1-2\eps)^{2/3} I_{\text{euc}}'(v_0) = (1-2\eps)\sqrt{\frac{16\pi}{(1-2\eps)^{2/3} I_{\text{euc}}(v_0)}}.
\]
This gives a contradiction to inequality (\ref{ineq3}). 
\end{proof}

As a corollary, we get a lower bound on the isoperimetric constant of $M$.

\begin{corollary}
\label{cor:iso} Let $M$ be as above. 
The isoperimetric constant $c_{\operatorname{iso}}$ for $M$ satisfies 
\[
c_{\operatorname{iso}}(M) \ge (1-2\eps)
^{2/3} (36\pi)^{1/3} = (1-2\eps)^{2/3} c_{\operatorname{iso}}(\R^3)
\]
where $\eps = \eps(\delta) \to 0$ as $\delta \to 0$. 
\end{corollary}

\section{Convergence to Euclidean in the $d_p$ metric}

\label{sec:dp}

In this section, we prove the $d_p$ convergence of the blow up models to Euclidean space.  This will complete the proof of Theorem \ref{main}.
First, we give some further background on the $d_p$ distance. 
Lee-Naber-Neumayer proved an $\eps$-regularity theorem for $d_p$. To understand the statement, we need to introduce Perelman's $\nu$-functional. 

\begin{defn}
For a smooth function $f$ on $(M^3,g)$ and a parameter $\tau > 0$, Perelman's $W$ functional is defined by 
\[
W(g,f,\tau) = \frac{1}{(4\pi \tau)^{3/2}} \int_M \bigg[\tau(\vert \grad f\vert^2 + R) + f - 3\bigg] e^{-f}\, dv. 
\]
Perelman's entropy is defined by 
\[
\mu(g,\tau) = \inf\left\{W(g,f,\tau): \frac{1}{(4\pi \tau)^{3/2}} \int_M e^{-f} = 1\right\},
\]
and Perelman's $\nu$-functional is given by $\nu(g,\tau) = \inf_{t\in (0,\tau)} \mu(g,\tau)$. 
\end{defn} 

It is known that $\mu(g,\tau) \le 0$ and that equality holds if and only if $(M,g)$ is isometric to Euclidean space. Roughly speaking, Lee-Naber-Neumayer's $\eps$-regularity theorem states that if $R(g) \ge -\delta$ and $\nu(g,2)\ge -\delta$ for some small $\delta > 0$ then $(M,g)$ must be $\eps$-close to Euclidean space with respect to the $d_p$ metric at unit $d_p$ scale. 
As explained in \cite{lee2023dp} Remark 1.9, one can in fact replace a lower bound on Perelman's $\nu$-functional with good control over the isoperimetric constant. For the reader's convenience, we prove the following proposition in detail. The proof can be seen as an instance of the P\'olya-Szeg\H{o} principle. 

\begin{prop}
\label{entropy}
Assume that $(M,g)$ is a complete, asymptotically flat 3-manifold with vanishing scalar curvature. Assume that the isoperimetric constant of $M$ satisfies 
\[
c_{\operatorname{iso}}(M) \ge \eta c_{\operatorname{iso}}(\R^3)
\]
for some number $0 < \eta < 1$. Then $\mu(g,\tau) \ge C \log\eta$ for all $\tau > 0$. 
\end{prop}

\begin{proof} 
We need to prove a lower bound on 
\[
W(g,f,\tau) = \frac{1}{(4\pi \tau)^{3/2}}\int_M (\tau \vert \grad f\vert^2 + f - 3)e^{-f}\, dv
\]
for all $f$ satisfying the constraint 
\[
\frac{1}{(4\pi \tau)^{3/2}}\int_M e^{-f}\, dv = 1. 
\]
Substituting $u^2 = e^{-f}$, it is equivalent to prove a lower bound for 
\[
\frac{1}{(4\pi \tau)^{3/2}} \int_M 4\tau \vert \grad u\vert^2 - 2 u^2\log u - 3u^2\, dv
\]
subject to the constraint 
\[
\int_M u^2\, dv = (4\pi \tau)^{3/2}. 
\]
Since the isoperimetric estimate is scale invariant, we can assume without loss of generality that $\tau = 1$. 

Therefore assume that $u$ is a positive function on $M$ with 
\[
\int_M u^2 \, dv = (4\pi)^{3/2}. 
\]
Let $\bar u$ be the radially symmetric function on $\R^3$ such that 
\[
\vol_M\{u \ge t\} = \vol_{\R^3}\{\bar u \ge t\}
\]
for all $t > 0$. We have 
\begin{gather*}
\int_M u^2 \, dv = \int_0^\infty t^2 \left(\int_{\{u=t\}} \vert \grad u\vert\inv \, da\right) \,dt\\
\int_{\R^3} \bar u^2 \, dv = \int_0^\infty t^2\left(\int_{\{\bar u=t\}} \vert \grad \bar u\vert\inv \, da\right) \,dt.
\end{gather*}
Thus since 
\[
\int_{\{u=t\}} \vert \grad u\vert\inv \, da = \frac{d}{dt} \vol\{u\ge t\} = \frac{d}{dt} \vol\{\bar u\ge t\} = \int_{\{\bar u=t\}} \vert \grad \bar u\vert\inv \, da,
\]
we obtain 
\[
\int_M u^2 \, dv = \int_{\R^3} \bar u^2\, dv.
\]
Likewise we also have 
\[
\int_M u^2 \log u \, dv = \int_{\R^3} \bar u^2 \log \bar u\, dv.
\]
It remains to estimate the term with the gradient. We have 
\begin{gather*}
\int_M {\vert \grad u\vert^2}\, dv = \int_0^\infty  \left(\int_{\{u=t\}} \vert \grad u\vert\, da\right)\, dt,\\
\int_{\R^3} {\vert \grad \bar u\vert^2}\, dv = \int_0^\infty  \left(\int_{\{\bar u=t\}} \vert \grad \bar u\vert\, da\right)\, dt.
\end{gather*}
Moreover, by H\"older's inequality and the isoperimetric estimate, we obtain 
\begin{align*}
\left(\int_{\{u =t\}} \vert \grad u\vert\, da \int_{\{u=t\}} \vert \grad u\vert\inv \,da\right)^{1/2} &\ge \area(\{u=t\}) \\
&\ge \eta \area(\{\bar u=t\})\\
&= \eta \left(\int_{\{\bar u =t\}} \vert \grad \bar u\vert\, da \int_{\{\bar u=t\}} \vert \grad \bar u\vert\inv \,da\right)^{1/2}.
\end{align*}
It follows that 
\[
\int_M {\vert \grad u\vert^2}\, dv \ge \eta^2 \int_{\R^3} {\vert \grad \bar u\vert^2}\, dv.
\]
Then combining everything we get that 
\[
W(g,f,1) \ge \frac{1}{(4\pi)^{3/2}}\int_{\R^3} 4\eta^2 \vert \grad \bar u\vert^2 - 2\bar u^2 \log \bar u - 3\bar u^2\, dv. 
\]
It remains to get a lower bound on the right hand side. 

Recall that $W(g_{\text{euc}},f,\tau)\ge 0$ for all $\tau > 0$ and all $f$ satisfying the constraint. We will apply this at a slightly different scale $\tau$ to get the desired inequality. Fix a constant $a > 0$ and define $\bar w = a\bar u$ and $\tau = a^{4/3}$. Then 
\[
\int_{\R^3} \bar w^2\, dv = (4\pi)^{3/2}a^2 = (4\pi\tau)^{3/2}. 
\]
Therefore, we have 
\begin{align*}
0 &\le \frac{1}{(4\pi)^{3/2}a^2} \int_{\R^3} 4 a^{4/3} \vert \grad \bar w\vert^2 - 2\bar w^2 \log \bar w - 3\bar w^2\, dv \\
&= \frac{1}{(4\pi)^{3/2}} \int_{\R^3} 4a^{4/3} \vert \grad \bar u\vert^2 - 2\bar u^2 \log (a\bar u) - 3\bar u^2\, dv \\
&=  \frac{1}{(4\pi)^{3/2}} \int_{\R^3} 4a^{4/3} \vert \grad \bar u\vert^2 - 2\bar u^2 \log \bar u - 2\bar u^2\log(a) - 3\bar u^2\, dv.
\end{align*}
In particular, we have 
\[
\frac{1}{(4\pi)^{3/2}} \int_{\R^3} 4a^{4/3} \vert \grad \bar u\vert^2 - 2\bar u^2 \log \bar u - 3\bar u^2\, dv \ge \int_{\R^3} 2\bar u^2 \log(a)\, dv. 
\]
Now set $a = \eta^{3/2}$ to get 
\[
 \frac{1}{(4\pi)^{3/2}}\int_{\R^3} 4\eta^2 \vert \grad \bar u\vert^2 - 2\bar u^2 \log \bar u - 3\bar u^2\, dv \ge  3\log \eta \int_{\R^3} \bar u^2 \, dv = 3 (4\pi)^{3/2} \log \eta. 
\]
Therefore we have 
$
W(g,f,1) \ge C\log \eta
$
for all $f$ satisfying the constraint and the result follows. 
\end{proof}

Finally we can give the proof of the first main result. 

\begin{rerethm}
Assume that $(M_i,g_i)$ is a sequence of complete, asymptotically flat 3-manifolds with vanishing scalar curvature. Assume that each $M_i$ is topologically $\R^3$. Further suppose that 
\[
\inf\left\{\frac{\int_{M_i} \vert \grad u\vert^2\, dv}{\left(\int_{M_i} u^6\, dv\right)^{1/3}}: u\in W^{1,2}(M_i)\right\} \ge \Lambda - \delta_i
\]
where $\delta_i\to 0$. Then $M_i$ converges to Euclidean space in the $d_p$ sense for all $p\in(3,\infty)$. 
\end{rerethm}

\begin{proof} Let $(M_i,g_i)$ be as in the statement of the theorem and fix some $p > 3$. 
According to Corollary \ref{cor:iso}, the isoperimetric constants of $M_i$ satisfy $c_{\operatorname{iso}}(M_i)\ge \eta_i c_{\operatorname{iso}}(\R^3)$ for some constants $\eta_i \to 1$ as $i\to \infty$. Hence Proposition \ref{entropy} implies that $\nu(g_i,2) \to 0$ as $i\to \infty$.  Moreover, we have $R(g_i)= 0$ for all $i$. It now follows that $(M_i,g_i)$ converges to Euclidean space in the $d_p$ sense by Theorem 1.7 in \cite{lee2023dp}.
\end{proof}

\section{Convergence of Conformal Factors} 
\label{further}

In this section, we show the convergence of the conformal factors and prove Theorem \ref{main2}. 
Let $(M_i,g_i)$ be a sequence of asymptotically flat manifolds as in the statement of Theorem \ref{main2} and consider the solutions $w_i$ to 
\[
\lap_{g_i} w_i + \lambda_i w_i^5 = 0
\]
satisfying 
\[
\|w_i\|_{L^\infty(M_i)} = 1 \text{ and } \|w_i\|_{L^6(M_i)}^6 = 2\pi^2 \text{ and } \lambda_i\to \frac{3}{4}. 
\]
Multiplying the equation by $w_i$ and integrating by parts, we see that $\grad w_i$ is also uniformly bounded in $L^2(M_i)$. For each $i$, choose a point $x_i\in M_i$ such that $w_i(x_i)=1$.  

Fix some value of $p$. Choose $r_j \nearrow \infty$ and $\eps_j\searrow 0$ and $\kappa_j\searrow 1$ and $s_j\nearrow \infty$. By Theorem 6.1 in \cite{lee2023dp}, after passing to a subsequence $(M_j)$ of $(M_i)$, we can find a neighborhood $\Omega_j$ of $x_j$ and a smooth diffeomorphism $\psi_j:\Omega_j \to B(0,r_j)\subset \R^3$ with $\psi_j(x_j) = 0$ such that the estimates
\begin{gather*}
(1-\eps_j) \|\psi_j^*f\|_{L^{q/\kappa_j}(\Omega_j)} \le \|f\|_{L^q(B(0,r_j))} \le (1+\eps_j)\|\psi_j^*f\|_{L^{\kappa_j q}(\Omega_j)},\\
(1-\eps_j)\|\grad \psi_j^*f \|_{L^{q/\kappa_j}(\Omega_j)} \le \|\grad f\|_{L^q(B(0,r_j))} \le (1+\eps_j)\|\grad \psi_j^*f\|_{L^{\kappa_j q}(\Omega_j)}. 
\end{gather*} 
hold for all $q\in(\kappa_j,s_j)$ and all $f\in W^{1,q}(B(0,r_j))$. The diffeomorphisms $\psi_j$ in particular are $\eps_j$-Gromov-Hausdorff approximations in the $d_p$ distance.  

Define the functions 
\[
\tilde w_j\colon B(0,r_j)\to \R
\]
by $\tilde w_j = w_j\circ \psi_j\inv$. 
Morally, the $d_p$ convergence of $M_j$ to Euclidean space implies that the $W^{1,2}$ analysis on $M_j$ is close to the $W^{1,2}$ analysis on Euclidean space. We want to use this to prove uniform estimates on the functions $\tilde w_j$.  We now prove a sequence of lemmas.

The first lemma gives the existence of good cut-off functions.  
Define $\mathcal D_j(r) = \psi_j\inv(B(0,r))$.

\begin{lem}
Fix $0< r \le 1$ and $0 < t \le 1$. For all sufficiently large $j$, there exists a cut-off function $\eta\ge 0$ defined on $M_j$ such that 
\begin{itemize}
\item[(i)] $\eta \equiv 1$ on $\mathcal D_j(r)$,
\item[(ii)] $\eta \equiv 0$ outside $\mathcal D_j(r+t)$,
\item[(iii)] $\|(\eta + \vert \grad \eta\vert)\|_{L^6(M_j)} \le Ct\inv$.
\end{itemize}
Here $C$ is a constant that does not depend on $r$, $t$ or $j$.
\end{lem}

\begin{proof} 
Let $\tilde \eta$ be a cut-off function on $\R^3$ with $\tilde \eta \equiv 1$ on $B(0,r)$ and $\tilde \eta \equiv 0$ outside $B(0,r+t)$. We can select $\tilde \eta$ so that 
\[
\vert \grad \tilde \eta\vert \le C t\inv. 
\]
It follows that 
\[
\|\grad \tilde \eta\|_{L^{6\kappa_j}(\R^3)}\le Ct^{\frac{1}{6\kappa_j}-1} \le C t^{-1},
\]
for $j$ large enough. 
Hence we obtain 
\begin{align*}
\| \grad \eta\|_{L^6(M_j)} \le \frac{1}{1-\eps_j} \|\grad \tilde \eta\|_{L^{6\kappa_j}(\R^3)} \le Ct\inv.
\end{align*}
Finally note that 
\[
\|\eta\|_{L^6(M_j)} \le \frac{1}{1-\eps_j} \|\tilde \eta\|_{L^{6\kappa_j}(\R^3)} \le C.
\]
Combining these observations gives the lemma. 
\end{proof}

Next, we use a Moser iteration scheme to prove uniform non-concentration estimates for the functions $w_j$. Here we closely follow the presentation of Gilbarg and Trudinger (see \cite{gilbarg1977elliptic} Section 8.6). 

\begin{lem}
\label{moser}
There are positive constants $c$ and $\sigma$ such that for any $r > 0$ and any $s \ge 2$, we have 
\[
\|w_j\|_{L^s(\mathcal D_j(2r))} \ge \left(\frac{c}{r^{\sigma}}\right)^{-1/s}
\]
for all sufficiently large $j$.
\end{lem}

\begin{proof}
In this proof, we drop the subscript $j$ so that $(M_j,g_j)$ is denoted by $(M,g)$, $w_j$ is denoted by $w$, and so on. Let $\rho_k = r(1+ 2^{-k})$ for $k\in \{0,1,2,\hdots\}$.  Choose a cut-off function $\eta$ according to the previous lemma which is 1 on $\mathcal D(\rho_k)$ and $0$ outside of $\mathcal D(\rho_{k-1})$. Since $0\le w\le 1$, we have 
$
\lap w = -\lambda w^5 \ge -\lambda w \ge -w
$
and so 
\[
\lap w + w \ge 0.
\]
Fix some $\beta \ge 2$. Multiply the above inequality by $\eta^2 w^{\beta}$ and integrate by parts to get 
\[
\int_{M} \eta^2 w^{\beta+1} - 2\eta w^{\beta} g(\grad \eta, \grad w) - \beta \eta^2 w^{\beta-1} \vert \grad w\vert^2 \, dv \ge 0.  
\]
Note that 
\begin{align*}
\vert 2\eta w^{\beta} g(\grad \eta, \grad w)\vert &\le 2 \eta w^\beta \vert \grad \eta\vert \vert \grad w\vert \le \eps  \eta^2 w^{\beta -1} \vert \grad w\vert^2 + \frac{1}{ \eps} w^{\beta + 1} \vert \grad \eta\vert^2 
\end{align*} 
for any $\eps > 0$. Selecting $\eps = \beta/2$, this yields 
\[
\int_{M}  \eta^2 w^{\beta-1} \vert \grad w\vert^2 \, dv \le C \int_{M} (\eta^2 + \vert \grad \eta\vert^2)w^{\beta+1}\, dv.
\]
Let $f = w^{(\beta + 1)/2}$.  Then, with $\gamma = \beta+1$, the above estimate implies that  
\[
\int_{M} \eta^2 \vert \grad f\vert^2\,dv \le C\gamma^2 \int_{M} (\eta^2 + \vert \grad \eta\vert^2)f^2\, dv. 
\]
Now, since $M$ supports a uniform $L^2$ Sobolev inequality, we get that 
\begin{align*}
\| \eta f\|_{L^6(M)}^2 &\le C \int_{M} \vert \eta \grad f\vert^2 + \vert f \grad \eta\vert^2\, dv\\
&\le C \gamma^2 \int_M (\eta^2 + \vert \grad \eta\vert^2)f^2 \, dv + C\int_M \vert \grad \eta\vert^2 f^2\, dv \\
&\le C (1+\gamma^2)\int_M (\eta + \vert \grad \eta\vert)^2 f^2\, dv,
\end{align*}
where $C$ does not depend on $M$. Hence one has 
\begin{align*}
\| f\|_{L^6(\mathcal D(\rho_k))} &\le C (1+\gamma^2)^{1/2} \|(\eta + \vert \grad \eta\vert) f \|_{L^2(M)}\\
&\le C(1+\gamma^2)^{1/2} \| (\eta + \vert \grad \eta \vert) \|_{L^6(M)} \|f\|_{L^3(\mathcal D(\rho_{k-1}))}. 
\end{align*}
Re-expressed in terms of $w$, this implies that 
\[
\| w \|_{L^{3\gamma}(\mathcal D(\rho_k))} \le C^{2/\gamma} (1+\gamma^2)^{{1}/{\gamma}} \|(\eta + \vert \grad \eta\vert)\|_{L^6(M)}^{2/\gamma} \|w\|_{L^{3\gamma/2}(\mathcal D(r_{\rho-1}))}.
\]
By the previous lemma, one can choose $\eta$ so that 
\[
\|(\eta + \vert \grad \eta\vert)\|_{L^6(M)} \le \frac{C\cdot 2^k}{r}.
\]
Let $\chi = 2$. Starting with $\gamma_0 = 2s/3$ and then iterating the previous estimate with $\gamma_k = \gamma_0\chi^k$, we get 
\[
\lim_{k\to \infty} \|w\|_{L^{3\gamma_k}(\mathcal D(\rho_k))} \le \|w\|_{L^s(\mathcal D(\rho_0))} \prod_{k=0}^\infty \left(\frac{C\chi^{k} 2^k}{r}\right)^{2/\gamma_k}.
\]
The limit on the left is equal to $\|w\|_{L^\infty(\mathcal D(r))} = 1$. Also, one has  
\[
\prod_{k=0}^\infty (C\chi^{k} 2^k)^{2/\gamma_k} = \left[\prod_{k=0}^\infty \left(\frac{C\cdot 4^{3k}}{r}\right)^{1/2^k}\right]^{1/p} = \left(\frac{c}{r^\sigma}\right)^{1/s},
\]
for some positive constants $c$ and $\sigma$. 
Therefore it follows that  
\[
\|w\|_{L^p(\mathcal D(2r))} \ge \left(\frac{c}{r^\sigma}\right)^{-1/s},
\]
and the lemma is proven. 
\end{proof} 

The final lemmas show that the $d_p$ convergence also gives control over the inner products of functions. 

\begin{lem}
\label{inner-prod}
Fix some large $j$. Consider a pair of $W^{1,2}$ functions $\tilde f_1$ and $\tilde f_2$ on $B(0,r_j)$ and let $f_1 = \tilde f_1\circ \psi_j$ and $f_2 = \tilde f_2\circ \psi_j$. Then the inequality
\[
\left\vert \int_{B(0,r)} g_{\text{euc}}(\grad \tilde f_1, \grad \tilde f_2)\, dv_{\text{euc}} - \int_{\Omega_j} g_j(\grad f_1, \grad f_2)\, dv_{g_j}\right\vert \le \eps_j \|\grad \tilde f_1\|_{L^2(B(0,r_j))} \| \grad \tilde f_2\|_{L^{2\kappa_j}(B(0,r_j))}
\]
holds.
\end{lem}

\begin{proof} Let $\phi_j = \psi_j\inv$. As in the proof of Theorem 6.1 in \cite{lee2023dp}, we can suppose that the set $\Omega_j$ admits a decomposition 
\[
\Omega_j = \bigcup_{k=1}^\infty \mathcal G^k\cup \mathcal A
\]
where 
\begin{gather*}
(1-\sigma_j)^{k+1}g_{\text{euc}} \le \phi_j^* g_j \le (1+\sigma_j)^{k+1} g_{\text{euc}}
\end{gather*}
for all $x\in \phi_j^* \mathcal G^k$, and moreover, 
\begin{gather*} 
\vol_{\text{euc}}(\phi_j^*\mathcal A) = 0,\\
\vol_{\text{euc}}(\phi_j^* \mathcal G^k) \le (1+\sigma_j)^k\sigma_j^{k-1},
\end{gather*}
for all $k\ge 2$. Here $\sigma_j$ is a very small number to be specified later.

We can now compute that 
\begin{align*}
&\left\vert \int_{B(0,r)} g_{\text{euc}}(\grad \tilde f_1, \grad \tilde f_2)\, dv_{\text{euc}} - \int_{\Omega_j} g_j(\grad f_1, \grad f_2)\, dv_{g_j}\right\vert \\
&\qquad = \left\vert \int_{B(0,r)} \delta^{lm} \bd_{l} \tilde f_1 \bd_{m}\tilde f_2 - (\phi^*g_j)^{lm} \bd_l \tilde f_1 \bd_m \tilde f_2 \sqrt{\det{\phi^* g_j}} \, dx\right\vert \\
&\qquad \le \int_{B(0,r)} \left\vert \delta^{lm} - (\phi^*g_j)^{lm}\sqrt{\det{\phi^*g_j}}\right\vert \vert \grad \tilde f_1\vert \vert \grad \tilde f_2\vert\, dx \\
&\qquad \le C \sum_{k=1}^\infty (1+\sigma_j)^{10k+1} \int_{\phi^*\mathcal G^k \cap B(0,r)} \vert \grad \tilde f_1\vert \vert \grad \tilde f_2\vert\, dx.
\end{align*}
Next choose $\kappa_j' = \kappa_j/(\kappa_j-1)$ and apply H\"older's inequality to get 
\begin{align*}
&\sum_{k=1}^\infty (1+\sigma_j)^{10k+1} \int_{\phi_j^*\mathcal G^k \cap B(0,r)} \vert \grad \tilde f_1\vert \vert \grad \tilde f_2\vert\, dx \\
&\qquad \le \sum_{k=1}^\infty (1+\sigma_j)^{10k+1} \left(\int_{\phi_j^*\mathcal G^k \cap B(0,r)} \vert \grad \tilde f_1\vert^2 \, dx\right)^{1/2} \left(\int_{\phi_j^*\mathcal G^k \cap B(0,r)} \vert \grad \tilde f_2\vert^2 \, dx\right)^{1/2}\\
&\qquad \le \sum_{k=1}^\infty (1+\sigma_j)^{10k+1} \vol_{\text{euc}}(\mathcal G^k)^{1/\kappa_j'}\|\grad \tilde f_1\|_{L^2(B(0,r))} \|\grad \tilde f_2\|_{L^{2\kappa_j}(B(0,r))}.
\end{align*} 
Finally observe that 
\[
\sum_{k=1}^\infty (1+\sigma_j)^{10k+1} \vol_{\text{euc}}(\mathcal G^k)^{1/\kappa_j'} \le \eps_j
\]
provided $\sigma_j$ is chosen small enough. 
\end{proof}

\begin{lem}
\label{inner-prod2} Fix some large $j$. Consider a pair of $L^6$ functions $\tilde f_1$ and $\tilde f_2$ on $B(0,r_j)$ and let $f_1 = \tilde f_1\circ \psi_j$ and $f_2 = \tilde f_2\circ \psi_j$. Then the inequality
\[
\left\vert \int_{B(0,r_j)} \tilde f_1^5 \tilde f_2 \, dv_{\text{euc}} - \int_{\Omega_j} f_1^5 f_2 \, dv_{g_j}\right\vert \le \eps_j \| \tilde f_1\|_{L^6(B(0,r_j))}^5 \|  \tilde f_2\|_{L^{6\kappa_j}(B(0,r_j))}
\]
holds.
\end{lem}

\begin{proof}
This is entirely analogous to the previous lemma. 
\end{proof}

We can now prove Theorem \ref{main2}. 

\begin{rerethm}
Assume that $(M_i,g_i)$ is a sequence of complete, asymptotically flat 3-manifolds with vanishing scalar curvature. Assume that each $M_i$ is topologically $\R^3$, and that the optimal constant in the $L^2$ Sobolev inequality on $M_i$ is approaching the Euclidean constant. Suppose that $M_i$ carries a positive solution to 
\[
\lap_{g_i} w_i + \lambda_i w_i^5 = 0
\]
with 
\[
\|w_i\|_{L^\infty(M_i)} = 1 \text{ and } \|w_i\|_{L^6(M_i)}^6 = 2\pi^2 \text{ and } \lambda_i \to \frac 3 4. 
\]
Let $\psi_i$ be the diffeomorphisms described above  and let $\tilde w_i = w_i \circ \psi_i\inv$. Then for any $q < 2$ and any $1 \le s < \infty$, the functions $\tilde w_i$ converge weakly in $W^{1,q}_{\text{loc}}(\R^3)$ and strongly in $L^s_{\text{loc}}(\R^3)$ to the function
\[
\tilde w(x) =  \sqrt{\frac{4}{4+\vert x\vert^2}}.
\]
\end{rerethm}

\begin{proof} Pass to a subsequence $(M_j)$ as indicated above. 
Note that $\grad w_j$ is uniformly bounded in $L^2(M_j)$. Moreover, for any $r > 0$, we have 
\[
\int_{\mathcal B_{p,g_j}(x_j,r)} w_j^2\, dv_j \le  \vol_{g_j}(\mathcal B_{p,g_j}(x_j,r))^{\frac 2 3} \left(\int_{\mathcal B_{p,g_j}(x_j,r)} w_j^6 \, dv_j\right)^{\frac1 3} \le C(r)
\]
by the $L^6$ bound on $w_j$ and the volume control coming from the $d_p$ convergence. Hence the functions $w_j$ have uniform bounds in $W^{1,2}_{\text{loc}}(M_j)$. It follows that 
$\tilde w_j$ has uniform $W^{1,q}_{\text{loc}}(\R^3)$ bounds 
for every $1 < q < 2$. Fix such a $q$ and assume that $q$ is sufficiently close to $2$. 
Passing to a subsequence if necessary, we can suppose that $\tilde w_j\to \tilde w$ for some function $\tilde w\colon \R^3\to \R$, where the convergence occurs weakly in $W^{1,q}$ and strongly in $L^{2}$ on every compact subset of $\R^3$. Passing to a further subsequence, we can suppose that $\tilde w_j \to \tilde w$ almost everywhere. 

We claim that $\|\tilde w\|_{L^\infty(B(0,4r))} = 1$ for all $r > 0$. In particular, $\tilde w$ is not the zero function. To see this, note that by dominated convergence and the $L^\infty$ bound, one has 
\[
\tilde w_j \to \tilde w \text{ in } L^s_{\text{loc}}(\R^3)
\]
for every $1 \le s < \infty$. Choose a cut-off function $\tilde \eta$ such that $\tilde \eta \equiv 1$ on $B(0,2r)$ and $\tilde \eta \equiv 0$ outside $B(0,4r)$. Let $\eta = \tilde \eta\circ \psi_j$. Then 
\begin{align*}
\|\tilde \eta \tilde w\|_{L^s(B(0,4r))} = \lim_{j\to \infty} \|\tilde \eta \tilde w_j\|_{L^s(B(0,4r))}.
\end{align*}
Lemma \ref{moser} implies that 
\begin{align*}
\|\tilde \eta \tilde w_j\|_{L^s(B(0,4r))} \ge \|\tilde \eta \tilde w_j\|_{L^{s\kappa_j}(\R^3)}^{\kappa_j} &\ge (1-\eps_j)^{\kappa_j} \|\eta w_j\|_{L^s(M_j)}^{\kappa_j} \\
&\ge (1-\eps_j)^{\kappa_j}  \|w_j\|_{L^s(\mathcal D_j(2r))}^{\kappa_j} \ge (1-\eps_j)^{\kappa_j}\left(\frac{c}{r^{\sigma}}\right)^{-\kappa_j/s}.
\end{align*}
Hence, sending $j\to\infty$, we deduce that 
\[
\|\tilde \eta \tilde w\|_{L^s(B(0,4r))} \ge \left(\frac{c}{r^\sigma}\right)^{-1/s}. 
\]
Now, letting $s\to \infty$, we deduce that 
\[
\|\tilde w\|_{L^\infty (B(0,4r))} \ge \|\tilde \eta \tilde w\|_{L^\infty(B(0,4r))} \ge 1.
\]
On the other hand, $\|\tilde w\|_{L^\infty(\R^3)}\le 1$ and the claim follows. 

Next, we claim that $\tilde w$ is a $W^{1,q}_{\text{loc}}$ weak solution of the equation 
\begin{equation}
\label{eq11}
\lap \tilde w + \frac 3 4 \tilde w^5 = 0
\end{equation}
in $\R^3$. It suffices to show that 
\begin{equation}
\label{eq10}
\int_{\R^3} -\grad \tilde w \cdot \grad \tilde u + \frac{3}{4} \tilde w^5 \tilde u\, dv = 0
\end{equation}
for all $\tilde u\in C^\infty_c(\R^3)$. So fix such a $ \tilde u$. We can suppose that $\tilde u$ is supported in $B(0,r)$. By the $W^{1,q}$ weak convergence and the strong $L^s$ convergence, we know that 
\[
\int_{\R^3} -\grad \tilde w \cdot \grad \tilde u + \frac{3}{4} \tilde w^5 \tilde u\, dv  = \lim_{j\to \infty} \int_{\R^3} -\grad \tilde w_j \cdot \grad \tilde u + \lambda_j \tilde w_j^5 \tilde u\, dv.
\] 
Now let $u_j = \tilde u \circ \psi_j$. Then, by the equation satisfied by $w_j$, one has 
\[
\int_{M_j} -g(\grad  w_j, \grad  u_j) + \lambda_j  w_j^5  u_j\, dv_j = 0. 
\]
Now by Lemma \ref{inner-prod}, it follows that 
\begin{align*}
&\left\vert \int_{\R^3} \grad \tilde w_j \cdot \grad \tilde u \, dv_{\text{euc}} - \int_{M_j} g(\grad  w_j, \grad  u_j)\, dv_{g_j} \right\vert \le \eps_j \|\grad \tilde w_j\|_{L^2(B(0,r))} \|\grad u\|_{L^{2\kappa_j}(B(0,r))} \to 0, \phantom{\int}
\end{align*}
as $j\to \infty$. 
Likewise, Lemma \ref{inner-prod2} implies that 
\[
\left\vert \int_{\R^3} \lambda_j \tilde w_j^5 \tilde u\, dv_{\text{euc}} - \int_{M_j} \lambda_j w_j^5 u_j\, dv_{g_j}\right\vert \le \eps_j \|\tilde w_j\|_{L^6(B(0,r))}^5 \|\tilde u\|_{L^{6\kappa_j}(B(0,r))} \to 0
\]
as $j\to \infty$.  Combining the previous four observations, we see that (\ref{eq10}) holds. By elliptic theory, it follows that $\tilde w$ is actually smooth. Combined with the fact that $\|\tilde w\|_{L^\infty(B(0,4r))}=1$ for all $r > 0$, we deduce that $\|\tilde w\|_{L^\infty(\R^3)} = 1$ and that $\tilde w(0) = 1$.

Next we claim that $\| \tilde w\|_{L^6(\R^3)}^6 \le 2\pi^2$. Suppose to the contrary that $\|\tilde w\|_{L^6(\R^3)}^6 \ge 2\pi^2 + 2\eta$. Then for large enough $j$ we have 
\[
\|\tilde w_j\|_{L^{6/\kappa_j}(B(0,r_j))}^{6/\kappa_j} \ge \|\tilde w_j\|_{L^6(B(0,r_j))}^6 \ge 2\pi^2 + \eta. 
\]
It follows that 
\[
\|w_j\|_{L^6(\Omega_j)} \ge \frac{1}{1+\eps_j} \|\tilde w_j\|_{L^{6/\kappa_j}(B(0,r_j))} \ge \frac{(2\pi^2 + \eta)^{\kappa_j/6}}{1+\eps_j}.
\]
After raising both sides to the sixth power, this contradicts the fact that $\|w_j\|_{L^6(M_j)}^6 = 2\pi^2$ for large $j$. 

Next we claim that $\grad w$ is globally in $L^2$. Indeed, multiply equation (\ref{eq11}) by $\eta^2 w$ where $\eta$ is a cut-off function to be specified later. This gives 
\begin{align*}
\int_{\R^3} \eta^2 \vert \grad w\vert^2\, dv &\le C\int_{\R^3} w^2\vert \grad \eta\vert^2 + \eta^2 w^6 \, dv \\
&\le C + C\left( \int_{\R^3} \vert \grad \eta\vert^3\, dv\right)^{2/3},
\end{align*}
where we used H\"older's inequality and the $L^6$ bound on $w$. Hence, selecting $r > 0$ and choosing $\eta$ so that $\eta\equiv 1$ on $B(0,r)$ and $\eta\equiv 0$ outside $B(0,2r)$ and $\vert \grad \eta\vert \le 2/r$ we get 
\[
\int_{B(0,r)} \vert \grad w\vert^2\, dv \le C
\]
where $C$ does not depend on $r$. Sending $r\to \infty$ gives the claim. 

Now multiply equation (\ref{eq11}) by $\tilde w$ and integrate to get 
\[
\int_{\R^3} \vert \grad \tilde w\vert^2 \, dv = \frac{3}{4} \int_{\R^3} \tilde w^6 \, dv,
\]
and therefore, by the Euclidean $L^2$ Sobolev inequality, that 
\begin{equation}
\label{eq12}
\Lambda \le \frac{\int_{\R^3} \vert \grad \tilde w\vert^2\, dv}{\left(\int_{\R^3} \tilde w^6\, dv\right)^{1/3}} = \frac{3}{4} \left(\int_{\R^3} \tilde w^6\, dv\right)^{2/3}. 
\end{equation}
It follows that 
\[
\|\tilde w\|_{L^6(\R^3)}^6 \ge \left(\frac{4\Lambda}{3}\right)^{3/2} = 2\pi^2,
\]
and therefore 
\[
\|\tilde w\|_{L^6(\R^3)}^6 = 2\pi^2
\]
and equality holds in (\ref{eq12}). Thus $\tilde w$ is a minimizer for the $L^2$ Sobolev inequality on $\R^3$. Such minimizers are classified (see \cite{lee1987yamabe} Section 3), and the only minimizer with 
\[
\|\tilde w\|_{L^6(\R^3)}^6 = 2\pi^2  \quad \text{and}\quad \|\tilde w\|_{L^\infty(\R^3)} = 1 \quad \text{and} \quad \tilde w(0) = 1 
\]
is
\[
\tilde w(x) =  \sqrt{\frac{4}{4+\vert x\vert^2}}.
\]
Since the limit $\tilde w$ does not depend on the subsequence chosen at the beginning, it follows that the full sequence converges to $\tilde w$. This completes the proof. 
\end{proof}

\bibliographystyle{plain}
\bibliography{reference}

\end{document}